\documentclass[titlepage,12pt]{article}
\usepackage{amsmath,amsthm,amssymb}
\usepackage[mathscr]{eucal}
\usepackage[latin1]{inputenc}
\usepackage{graphicx}
\usepackage{rotating}
\usepackage{soul}
\newtheorem{The}{Theorem}

\newtheorem{Exam}[The]{Example}
\newtheorem{Cox}[The]{Counterexample}
\theoremstyle{definition}

\numberwithin{equation}{section} \numberwithin{The}{section}

\usepackage{color}

\begin{document}

\title{Comparisons of policies based on relevation and replacement by a new one unit in reliability}

\author{
F\'{e}lix Belzunce\footnote{Corresponding author}$^{,2}$, Carolina Martínez-Riquelme$^2$\\
José A. Mercader$^3$,  José M. Ruiz$^2$ \\ $^2$Dpto.
Estad\'{\i}stica e Investigaci\'{o}n Operativa \\ Universidad de Murcia
\\ Facultad de Matem\'{a}ticas, Campus de Espinardo \\
30100 Espinardo (Murcia), SPAIN\\ $^3$Dpto. Matemática Aplicada y Estadística \\ Universidad Politécnica de Cartagena \\ Paseo Alfonso XIII, 52\\ 30203 Cartagena, SPAIN\\ e-mail: \texttt{belzunce@um.es, carolina.martinez7@um.es} \\  \texttt{josea.mercader@upct.es, jmruizgo@um.es}
}
\date{}

\maketitle

\begin{abstract}  
The purpose of this paper is to study the role of the relevation transform, where a failed unit is replaced by a used unit with the same age as the failed one, as an alternative to the policy based on the replacement by a new one. In particular, we compare the stochastic processes arising from a policy based on the replacement of a failed unit by a new one and from the one in which the unit  is being continuously subjected to a relevation policy. The comparisons depend on the ageing properties of the units under repair. 


\end{abstract}

\section{Introduction}
The concept of relevation was introduced by Krakowsky (1973) and was further explored by Baxter (1982) and Shanthikumar and Baxter (1985). This operation, among two survival functions,  provides a survival function arising in the following situation. Let us consider two items with independent non negative random lifetimes $T$ and $S$, and survival functions $\overline F(t)=P[T>t]$ and $\overline G(t)=P[S>t]$, respectively. The item with random lifetime $T$ is replaced, upon failure, by an item with random lifetime $S$, but with the same age as the first unit. Then, the survival function of the next failure time is given by
\begin{equation}\label{eq-rel}
\overline F(t) + \overline G(t) \int_0^t \frac{dF(x)}{\overline G(x)},\text{ for } t\ge 0.
\end{equation}

This operation  is known as the relevation transform and will be denoted by $T \# S$. This operation is clearly non commutative and, denoting the residual lifetime of $S$ at time $t$ by
\[
S_t\equiv [S-t|S>t],\text{ for all }t, \text{ such that }\overline G(t)>0,
\]
it is satisfied that
\begin{equation}\label{eq1.2}
T \# S = T + S_T,
\end{equation}
assuming that $Supp(T) \subseteq Supp(S)$, $T$ and $S$ are independent and $T<S$ (a.s.). Recall that this operation is different from that of the residual lifetime at random time. In this case, if $T$ and $S$ are two independent non negative random variables, then the random variable $[T - S|T > S]$ is the residual lifetime of $T$ at random time $S$. The relevation can be also used to describe the following repair policy. Let us consider that a failed system is repaired by replacing or repairing one or more of its components, then  the  age of the repaired system is the same of the system prior to the failure, but the underlying distribution has changed given that some of its components are different. Therefore, the relevation can be considered not only as a replacement policy, but also as a repair policy, similar to the minimal repair policy where, under a failure, the system is restored to a working condition just previous to the failure, in the sense that the age of the system has not changed, but the system failure rate does. This may occur when some important components of the system are replaced or repaired.  Recall that the relevation has been used in other contexts, see Pellerey, Shaked and Zinn (2000) and Belzunce, Lillo, Ruiz and Shaked (2001) and has been considered more recently in the context of coherent systems, see Belzunce, Martínez-Riquelme and Ruiz (2019) and Navarro, Arriaza and Suárez-Llorens (2019).

The relevation policy is clearly an alternative to the one in which an item, upon failure, is replaced by a new one. Dealing with an item which is being continuously under a relevation policy, as an alternative to the replacement by a new one  policy, the following question arises: Is the number of replacements under a relevation policy greater or smaller than the number of replacements by a new one, by any time $t$? Alternatively, this question can be rephrased in terms of the number of failures or the failure times. 

The comparison of repair and/or replacement policies in terms of the number of failures or  failure times has been already considered in the literature. The earlier results on the topic can be found in Barlow and Proschan (1975), Langberg (1988),  Block, Langberg and Savits (1990) and (1993) and Baxter, Kijima and Tortorella (1996).  In these papers, the authors compare the number of failures and removals for any fixed interval time $[0,t]$ under different ageing properties of the underlying distributions. These comparisons are made in the usual stochastic order and they can be applied under  the NBU (NWU) ageing property of the corresponding units.  Kijima, Li and Shaked (2001) summarize a great number of these results and, additionally, Li (2005) provides ``several maintenance comparison results on the age-dependent block policy", as well as a  survey on previous results for some other repair policies. 

Additionally, Belzunce, Lillo, Ruiz and Shaked (2001) provide comparisons for two relevation counting processes, Belzunce, Ortega and Ruiz (2005) and (2006) give some results for the comparison of renewal, age and block replacement policies, and more recently, Badía, Sangüesa and Cha (2018a, 2018b) obtain new results for the comparison of trend renewal and generalized Polya processes, which include minimal repair processes and some other policies, as a particular case, and Mercier and Castro (2019) compare the number of failures and failure times for two imperfect repair models for a degradation system.

We want to mention that some other authors compare replacement policies in view of cost rates and optimization (see Nakagawa and Zhao,  2013, Zhao, Mizutani and Nakagawa, 2015 and Zhao, Quian and Nakagawa, 2017, and the references therein). This approach has not been considered in this paper. 

This work is in the spirit of the first group of the above mentioned results, and it is intended  to provide new results which compare the sequence of failures times of the policies previously described. In these results, we will provide conditions under which the sequence of failure times under one of the policies is bigger, in some probabilistic sense, than the sequence under the other policy. It is important to notice that the new results are given in a multivariate setting, whereas most of the previous results are given in the univariate setting, and the new ones are not only for the stochastic order but also for the multivariate hazard rate order. 

In order to give the new results, we will provide several properties of a relevation process in Section \ref{Sec2},  and we will define the main tools to compare failure times, which are the stochastic and hazard rate orders and the IFR [DFR] and NBU [NWU] ageing notions. Later, in Section \ref{Sec3}, we will provide several results on comparing sequences of failure times under the above mentioned policies. In Section \ref{Sec4} we discuss some applications of these results for minimal repair and generalized Yule birth processes. To finish, we make a section with some conclusions.

Along the paper, we will use the following notation. By ``increasing" and ``decreasing" we mean ``nondecreasing" and ``nonincreasing", respectively. For any random variable $X$ and any event $A$, we use $[X|A]$ to denote the random variable whose distribution is the conditional distribution of $X$ given $A$. In addition, given a random variable $X$ with continuous distribution function $F$ and density $f$, then, $\bar{F}\equiv 1-F$ denotes the survival function and $f/\overline{F}$ is the hazard rate function associated to $X$. Finally, we denote by $\overset{\text{d}}{=}$ the equality in law.

\section{Preliminaries on relevation, stochastic comparisons and ageing properties}\label{Sec2}
In this section, we describe first the structure of the relevation policy as a counting  process. To do this, we recall the definition of an elementary pure birth process. 

A counting process $\{N(t);t\ge 0\}$ is called an elementary (or non homogeneous) pure birth (EPB) process  if (Pfeifer, 1982b)

\begin{itemize}

\item[(a)] $\{N(t); t\ge 0\}$ possesses right continuous paths;

\item[(b)] the limits 
\[
\lambda_n(t)= \lim_{h \rightarrow 0^+}\frac{1}{h} P(N(t+h)=n+1|N(t)=n),\text{ for all }n\ge 1, t\ge 0
\]
are positive and finite; and 

\item[(c)] for the sequence of arrival times $T_n=\sup\{t\ge 0|N(t)=n\}$, $n\ge1$ all finite dimensional marginals are absolutely continuous with respect to Lebesgue measure. 
\end{itemize}

The sequence of arrival times has been shown to be identically distributed with the record sequence of independent random variables when the underlying random variables may change after each record value (see Pfeifer, 1982b). Some of its stochastic properties are the following (see Pfeifer, 1982a, 1982b, Kamps, 1995 and Torrado, Lillo and Wiper, 2012): 

\begin{itemize}

\item[(P1)] $\{T_n; n\ge1\}$ is a Markov chain with transition probabilities
\begin{equation}\label{mc-ebp}
P(T_n>t|T_{n-1}=s)=\frac{1-F_n(t)}{1-F_n(s)},\text{ for }0\le s \le t,
\end{equation}
where $F_n(x)=1-\exp{\int_0^x \lambda_n(u)du}$ is a distribution function for a non negative random variable, that we will denote by $X_n$, for all $n\ge1$. Therefore, $P(T_n>t|T_{n-1}=s)=P(X_n>t|X_n>s)$, for $0\le s \le t$. 

\item[(P2)] For $n\ge 1$, the joint density function of $(T_1,T_2,\ldots, T_n)$ is given by
\[
f(t_1, t_2, \ldots,t_n)=\prod _{i=1}^{n-1} \frac{f_i(x_i)}{1-F_i(x_i)}f_n(x_n)\text{ for }0< t_1 < t_2 < \cdots < t_n,
\]
where $f_i$ is the density function corresponding to $F_i$. Therefore an EPB process is completely defined by the sequence of distribution functions $\{F_n; n\ge 1\}$. 
\end{itemize}

Properties P1 and P2 characterize the EPB process. Also, they can be used to provide the marginal distributions of the arrival times. It can be seen that, for $n\ge 1$, the marginal survival function of $T_n$ is given recursively by
\begin{equation}\label{dist-t1}
\overline G_1(t) = \overline F_1(t),\text{ for }t\ge 0, 
\end{equation}
and
\begin{equation}\label{dist-tn}
\overline G_n(t) =  \overline G_{n-1}(t) + \overline F_{n}(t) \int_0^t \frac{d G_{n-1}(x)}{\overline F_n(t)},\text{ for }t\ge 0.
\end{equation}

Now, we observe the relationship of an EPB process with the sequence of failures according to a relevation policy. Let us consider a sequence of distribution functions $\{F_n; n\ge 1\}$. Let us assume that we consider initially an item with distribution function $F_1$, and let us denote by $T_1$ the corresponding failure time. Given $T_1=t_1$, we perform a relevation of the first item with another item, stochastically independent of the previous one, with the same age of the failed item and with a baseline distribution $F_2$, that is, given $T_1=t_1$ the survival function of the next failure, denoted by $T_2$,  is given by $(1-F_2(t))/(1-F_2(t_1))$ for $0\le t_1 \le t$. In general, given a sequence of failure times $T_1=t_1<\cdots<T_{i-1}=t_{i-1}$, we perform a relevation of the failed item with an item, stochastically independent of the previous ones,  with age $t_{i-1}$ and baseline distribution $F_i$, that is, the survival function of the next failure time $T_i$ is given by $(1-F_i(t))/(1-F_i(t_{i-1}))$ for $0\le t_{i-1} \le t$. From P1, we observe that the sequence of failure times, $\{T_n;n\ge 1\}$, correspond with the arrival times of an EPB process where the sequence $\{\lambda_n;n\ge1\}$ is given by $\lambda_n = f_n/\overline F_n$, for $n\ge 1$. This fact has been observed, under different points of view, by Cramer and Kamps (2003), Lenz (2008) and Torrado, Lillo and Wiper (2012). Therefore, P2 gives the joint distribution of the failure times $(T_1,\ldots,T_n)$ and equations \eqref{dist-t1} and \eqref{dist-tn} give the marginal distributions, where we can see the relevation transform of the corresponding distribution functions. 

As we will see in Section \ref{Sec4}, two important stochastic models in reliability can be described as a particular case of a proper EPB process. In particular, minimal repair policies and generalized Yule birth processes are particular cases of this counting process.

 As mentioned in the introduction, in contrast to the relevation policy, we can consider a policy, in which, upon failure, the item is replaced by a new one, stochastically independent of the previous ones. In particular, given a sequence of independent, but not necessarily identically distributed, continuous random variables $\{X_n;n\ge 1\}$, the arrival times of this policy are given by 
\[
T'_n \equiv X_1+\cdots + X_n,\text{ for all }n\ge 1.
\]
 
The corresponding counting process will be denoted by $\{N'(t); t\ge 0\}$, where $N'(t) = \sup\{n |T'_n \le t\}$, for $t\ge 0$. Given that the random variables $X_n$ are independent, the sequence of arrival times is a Markov chain with transition probabilities 
\begin{equation}\label{mc-ren}
P(T'_n>t|T'_{n-1}=s)=1-F_n(t-s),\text{ for }0\le s \le t,
\end{equation}
where $F_n$ is the distribution function of $X_n$, for $n\ge1$. Therefore, $P(T'_n>t|T'_{n-1}=s)=P(X_n+s>t)$, for $0\le s \le t$.

Therefore, as mentioned in the introduction, our main goal is to provide comparisons of the two previous counting processes.  There are several criteria to compare two random quantities, not only in reliability, but also in risk theory, finance and economy. In reliability, these comparisons are established in terms of some functions which are of interest to describe the ageing process of a unit or a system. For a detailed introduction and discussion see Müller and Stoyan (2002), Shaked and Shanthikumar (2007) and Belzunce, Martínez-Riquelme and Mulero (2016).

In particular, most results considered along this paper are given in terms of the usual stochastic order. Given two random variables $X$ and $Y$, with survival functions $\overline F$ and $\overline G$, we say that $X$ is smaller than $Y$ in the stochastic order, denoted by $X\le_{\textup{st}}Y$, if
\[
\overline F(t) \le \overline G(t),\text{ for all }t\in\mathbb R.
\]

Clearly, the stochastic order formalizes the idea that the random variable $X$ is less likely to take on large values than $Y$. For instance, when $X$ and $Y$ are considered as random lifetimes or random times at which some event occurs, $Y$ has larger lifetimes than $X$, in a probabilistic sense. An interesting characterization of the stochastic order is the following
\begin{equation*}\label{eq2.5}
X\le_{\textup{st}}Y\text{, if and only if, }E[\phi(X)]\le E[\phi(Y)],
\end{equation*}
for all increasing function $\phi$ such that the previous expectations exist.

The previous characterization has been used to provide a multivariate extension of the usual stochastic order. Given two random vectors $\mathbf X=(X_1,\ldots,X_n)$ and $\mathbf Y=(Y_1,\ldots, Y_n)$, we say that $\mathbf X$ is smaller than $\mathbf Y$ in the multivariate stochastic order, denoted by $\mathbf X \le_{\textup{st}} \mathbf Y$, if 
\[
E[\phi(\mathbf X)]\le E[\phi(\mathbf Y)],
\]
for all increasing function $\phi:\mathbb R^n \mapsto \mathbb R$ such that the previous expectations exist. As a direct consequence of this definition, we have that if $\mathbf X \le_{\textup{st}} \mathbf Y$, then the marginal components of the random vector are also ordered in the stochastic order, that is, $X_i \le_{\textup{st}} Y_i$, for $i=1,\ldots,n$.

An additional stochastic order will be used to state some of the results. We say that $X$ is smaller than $Y$ in the hazard rate order, denoted by $X\le_{\textup{hr}} Y$, if 
\[
\frac{\overline{G}(x)}{\overline{F}(x)} \text{ is increasing in }x\text{ over } \left\{x\in \mathbb R | \overline F(x)>0\right \}.
\]

The hazard rate order can be also characterized in terms of the hazard rate function. If the distribution functions are differentiable, denoting by $r$ and $s$ the hazard rates of $X$ and $Y$, respectively, then $X\le_{\textup{hr}}Y$ if, and only if, 
\[
r(t)\ge s(t)\text{, for all }t\text{ such that }\overline F(t), \overline G(t)>0.
\]

Based on the interpretation of the hazard rate function in reliability, it is clear that the hazard rate order plays a similar role in reliability as the stochastic order.

Next, we provide a dynamic multivariate version of this order (see Shaked and Shanthikumar, 2007 and Belzunce, Martínez-Riquelme and Mulero, 2016). 

Let us consider a random vector $\mathbf{X}=(X_1,\dots,X_n)$ where the $X_i$'s represent the lifetimes of $n$ units. Therefore, the components are assumed to be non negative. For $t\ge 0$, let $h_t$ denote the list of units which have failed and their failure
times, which is called a \textit{history}. Down to the last detail,
\[
h_t=\{\mathbf{X}_I=\mathbf{x}_I,\mathbf{X}_{\overline{I}}>t\mathbf{e}\},
\]
where $I=\{i_1,\dots,i_k\}\subseteq\{1,\dots,n\}$,
$\overline{I}=\{1,\dots,n\}\setminus I$,
$\mathbf{X}_I$ denotes the vector formed by the components of
$\mathbf{X}$ with index in $I$ and $0<x_{i_j}<t$, for all
$j=1,\dots,k$, and $\mathbf{e}$ denotes a vector of 1's, where the
dimension is determined from the context. In this case, the dimension of $\mathbf{e}$ is equal to $n-k$.

Given a history $h_t$ as above and $j\in \overline{I}$, the
\textit{multivariate dynamic hazard rate function} of $X_j$ given the history $h_t$, is defined by
\begin{equation}\label{mchr} \eta_{j}(t|h_t)=\lim_{\Delta  \rightarrow
0^+}\frac{1}{\Delta }P(t<X_j\le t+\Delta |h_t),\text{ for all }t\ge 0.
\end{equation}
Clearly, $\eta_{j}(t|h_t)$ is the intensity of failure of the component $j$, given the history $h_t$, or the failure rate of $X_j$ at time $t$ given $h_t$.

In case $X_1<\cdots<X_n$ (a.s.), given $t > 0$, it is clear that any history $h_t$ has the following form
\[
h_t=\{X_1=x_1, \ldots,X_k=x_k, X_{k+1}>t,\ldots,X_{n}>t\},
\] 
where $0 < x_1<\cdots<x_k<t$, for some $k\in\{0,1,\ldots,n\}$.

Given two non negative continuous random vectors $\mathbf{X}=(X_1,\dots,X_n)$ and $\mathbf{Y}=(Y_1,\dots,Y_n)$ with multivariate dynamic hazard rate functions $\eta_\cdot(\cdot|\cdot)$ and $\lambda_\cdot(\cdot|\cdot)$, respectively, we say that $\mathbf X$ is smaller than $\mathbf Y$ in the \textit{multivariate dynamic hazard rate order},
denoted by $\mathbf X \leq_{\textup{dyn-hr}}\mathbf Y$, if
\[
\eta_{k}(t|h_t)\ge\lambda_{k}(t|h'_t),\text{ for all }t\ge 0,
\]
for all
\begin{equation*}\label{history1}
h_t=\{\mathbf{X}_{I\cup J}=\mathbf{x}_{I\cup
J},\mathbf{X}_{\overline{I\cup J}}>t\mathbf{e}\}
\end{equation*}
 and
\begin{equation*}\label{history2}
h'_t=\{\mathbf{Y}_I=\mathbf{y}_I,\mathbf{Y}_{\overline{I}}>t\mathbf{e}\},
\end{equation*}
whenever $I\cap J=\emptyset$,
$\mathbf{0}\leq\mathbf{x}_I\le\mathbf{y}_I\le t\mathbf{e}$, $\boldsymbol{0}\leq\boldsymbol{x}_J\le t\mathbf{e}$, and for all $k\in\overline{I\cup J}$. In this case, we say that $h_t$ is \textit{more severe} than
$h'_t$.

It is well known that (see Shaked and Shanthikumar, 2007 and Belzunce, Martínez-Riquelme and Mulero, 2016)
\[
\mathbf X \leq_{\textup{dyn-hr}}\mathbf Y \Rightarrow \mathbf X \leq_{\textup{st}}\mathbf Y.
\]

To finish, we provide the definition of two ageing properties that will be used along the next section (see Barlow and Proschan, 1975 and Lai and Xie, 2006).

A random variable $X$ with hazard rate $r$ is said to be IFR [DFR] (increasing [decreasing] failure rate), if $r$ is an increasing [decreasing] function. In addition, $X$ is said to be NBU [NWU] (new better [worst] than used), if
\[
X\ge_{\textup{st}}[\le_{\textup{st}}]\{X-t|X>t\}\text{ for all }t\text{ such that }\overline F(t)>0.
\]

It is well known that 
\[
X\text{ is IFR [DFR]}\Rightarrow X\text{ is NBU [NWU]}.
\]

\section{Comparison of policies based on relevation and replacement by a new one unit}\label{Sec3}

According to the introduction, the relevation policy is a clear alternative to the replacement of a failed item with a new one. Under the notation of previous section, we have two counting processes, $\{N(t);t\ge 0\}$ and $\{N'(t);t\ge 0\}$, respectively, defined by the sequence of these failure times. In this section, we consider the comparison of these two counting processes. As a starting point, we formulate the following question: Is the number of replacements under a relevation policy greater or smaller, than the one under a replacement by a new one unit?  Formally, the previous question can be stated in the following terms:  Can we assure that $N'(t)\le_{st}(\ge_{st}) N(t)$, for all $t>0$?

It is well known that, for a general counting process $\{N(t);t\ge 0\}$ with arrival times $T_n = \sup\{t \ge 0 | N(t)=n\}$, the following equivalence holds
\[
\{N(t)<n\} \Leftrightarrow \{T_n> t\}.
\]

Therefore, the comparison of the number of replacements for any fixed interval of time is equivalent to the comparison of arrival times, that is, 
\[N'(t)\le_{st}(\ge_{st}) N(t),\text{ for all }t\ge0 \Leftrightarrow T_n \ge_{st}(\le_{st})T'_n,\text{ for all }n\ge 1. 
\]

Every kind of result in this direction would help an engineer to select a policy with the smallest number of replacements and, equivalently, the greatest failure times. In general, the result does not hold as we show next.

It is not difficult to prove that, if we consider two independent random variables $T_1$ and $T_2$ with common distribution $F$, then $T_1 \# T_2 \le_{\textup st} (\ge_{\textup st}) T_1 + T_2$, if, and only if, 
\[
\int_0^t \left( \overline F(t-x) - \frac{\overline F(t)}{\overline F(x)} \right) dF(x) \ge (\le) 0,\text{ for all }t\ge 0.
\]
Clearly, this integral is positive (negative) if $F$ is NBU (NWU). This is the case of a gamma distribution, that is, with density function given by 
\[
f(x)=\frac{x^{a -1}\exp\left\{-\frac{x}{b}\right\}}{ b^{a} \Gamma(a)},\text{ for all }0<x<+\infty,
\]
where $a,b>0$. Since the gamma distribution is NBU (NWU) if the shape parameter $a$ is greater (smaller) than 1, given that a used unit is stochastically smaller (greater) than a new one, a relevation should lead to an earlier (later) time for the second failure than the replacement by a new one. The NBU (NWU) condition appears in many other results as a sufficient condition for the comparison of the number of failures and failure times among renewal processes and some other repair policies like age and block policies (see Kijima, Li and Shaked, 2001 and Li, 2005). 

Therefore, there are situations where  either $T'_n \ge_{st}T_n$ or $T'_n\le_{st}T_n$. 

It is worth to notice that the stochastic order does not hold when we remove the NBU or NWU ageing conditions. Let us consider the following counterexample. 

\begin{Cox} Let $T_1$ and $T_2$ be two independent random variables, with common distribution function 
\[
F(t)=1-\exp(-t^{0.2}e^{1.1t}),\text{ for all }t\ge0.
\]

It is known (see Lai and Xie, 2006, p. 79) that the hazard rate is non monotone. Actually, $F$ is not either NBU or NWU. In Figure \ref{cox} we can observe that the survival functions corresponding to $T_1 \# T_2$ and $T_1 + T_2$ cross each other, and the stochastic order does not hold, in any sense, for the two random variables. 

\begin{figure}[hbt]\centering
\includegraphics[scale=0.55]{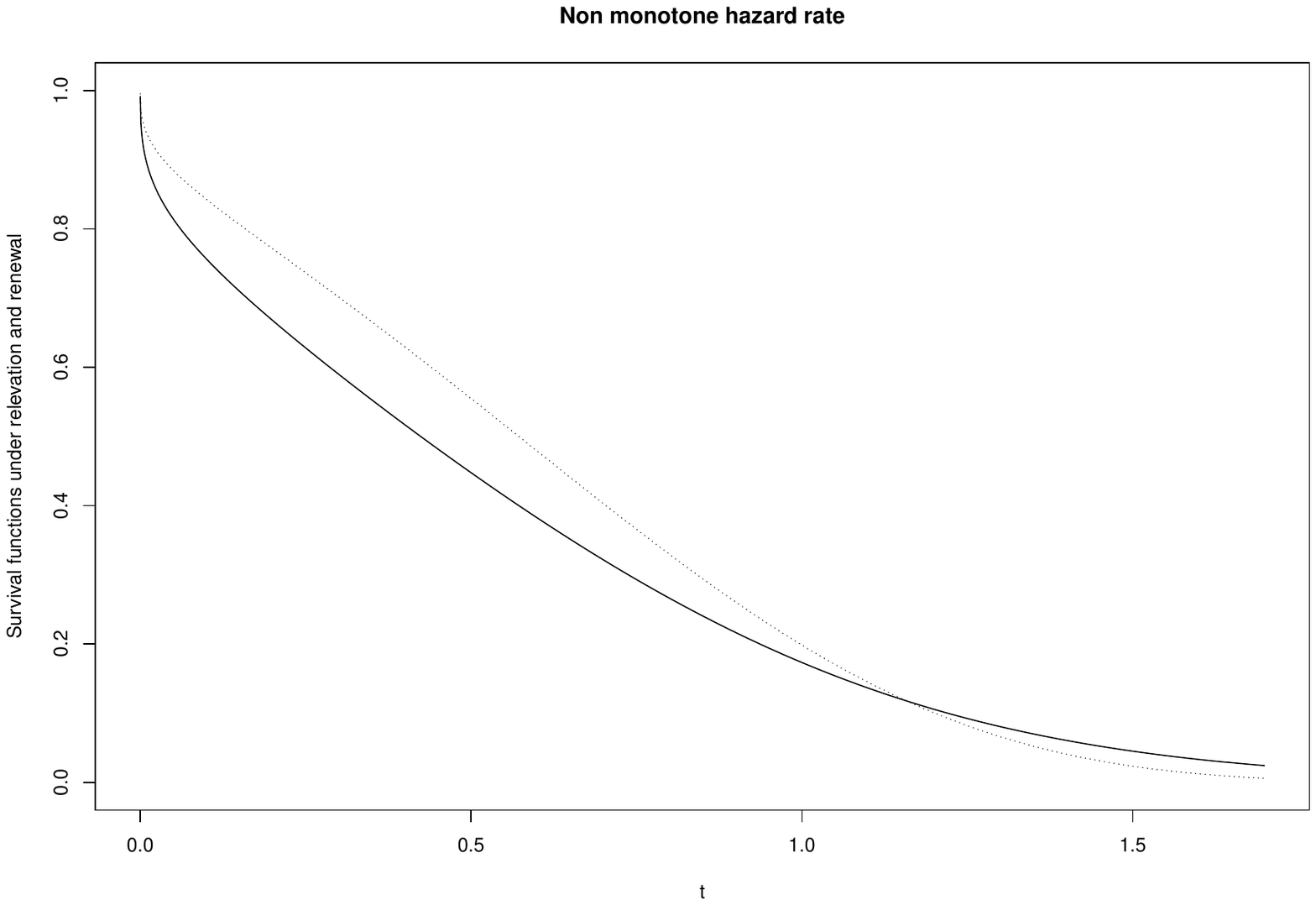}
\caption{\label{cox}\textit{Plot of the survival functions of $T_1 \# T_2$ (dashed line) and $T_1 + T_2$ (continuous line).}}
\end{figure}
\end{Cox}

Therefore, it is natural to consider  the NBU (NWU) property as a sufficient condition for the comparison of the two previous policies. 

In the next theorem we prove this result in a more general setting. In particular, we compare the arrival times of the stochastic processes $N$ and $N'$, in terms of the stochastic order. In addition, we get univariate comparisons for the arrival times and the number of replacements as a consequence of these multivariate comparisons. All the random variables considered along the following sections are considered to be non negative.

Along the proof, the following definition of conditionally increasing in sequence property is required. Given a random vector $\mathbf{X}=(X_1,\dots,X_n)$, we say that $\mathbf{X}$ is \textit{conditionally increasing in sequence}, denoted by CIS, if the distribution function of the random variable
\begin{equation*}
[X_i| X_1=x_1,\dots,X_{i-1}=x_{i-1}]
\end{equation*}
is decreasing in $(x_1,\dots,x_{i-1})$ in the support of $(X_1,\ldots,X_{i-1})$, for all  $i=2,\dots,n$.

\begin{The}\label{Th3.2} Let us consider the processes,  $N$ and $N'$, based on a sequence of distributions functions $\{F_n;n\ge 1\}$, corresponding to a sequence of independent continuous random variables, $\{X_n; n\ge 1\}$, with arrival times $\{T_n;n\ge 1\}$ and $\{T'_n; n\ge 1\}$, respectively. If $X_i$ is  NBU \textup{[}NWU\textup{]}, for all $i\ge 1$, then, 
\[
(T'_1,T'_2, \ldots,T'_n)\ge_{\textup{st}}[\le_{\textup{st}}](T_1,T_2, \ldots,T_n)\text{ for all }n\ge 1,
\]
and 
\[
N'(t) \le_{\textup{st}} [\ge_{\textup{st}}] N(t)\text{, for all }t\ge 0.
\]
\end{The}

\begin{proof} Let us assume that $X_i$ is  NBU, for all $i\ge 1$, the NWU case follows similarly by reversing the inequalities. According to Theorem 6.B.4 in Shaked and Shanthikumar (2007) (see also Theorem 3.2.4 in Belzunce, Martínez-Riquelme and Mulero, 2016), to get the results it is sufficient to prove,  that, for any $n>1$, 
\begin{equation}\label{eq3.1}
T_1 \le_{st} T'_1,
\end{equation}
\begin{equation}\label{eq3.2}
[T_i|T_1=t_1,\ldots,T_{i-1}=t_{i-1}]\le_{\textup{st}}[T'_i|T'_1=t_1,\ldots,T'_{i-1}=t_{i-1}]
\end{equation}
for all $0<t_1<\cdots<t_{i-1}$ and for all $i=2,\ldots,n$, and
\begin{equation}\label{eq3.3}
(T_1,\ldots,T_n)\text{ or } (T'_1,\ldots,T'_n),\text{ or both, are CIS}.
\end{equation}

Condition \eqref{eq3.1} is trivial given that $T_1 \overset{\text{d}}{=} T'_1 \overset{\text{d}}{=} X_1$. According to Belzunce, Lillo, Ruiz and Shaked (2001), $(T_1,\ldots,T_n)$ satisfies \eqref{eq3.3}.  Therefore, it just remains to prove \eqref{eq3.2}. 

From the Markovian property of these two processes, and equations \eqref{mc-ebp} and \eqref{mc-ren}, we have that  
\[
[T_i|T_1=t_1,\ldots,T_{i-1}=t_{i-1}]\overset{\text{d}}{=}[X_i|X_i>t_{i-1}]
\]
and 
\[
[T'_i|T'_1=t_1,\ldots,T'_{i-1}=t_{i-1}]\overset{\text{d}}{=} X_i+t_{i-1},
\]
consequently, \eqref{eq3.2} is equivalent to 
\[
[X_i-t_{i-1}|X_i>t_{i-1}]\le_{st} X_i,\text{ for all }t_{i-1}>0,
\]
which is equivalent to $X_i$ being NBU. 

To finish, we observe that, from this result, we get that $T'_n \ge_{\textup{st}} T_n$, for all $n\ge 1$, and therefore $N'(t) \le_{\textup{st}}  N(t)$, for all $t\ge 0$.
\end{proof} 

Therefore, the relevation policy gives rise to more failures than the replacement by a new one by any time $t$,  although it has a lower cost than the replacement by a new one policy.

It is natural to consider some other cases under different ageing properties of the unit under replacement. We have considered this problem when the NBU [NWU] property is replaced by the IFR [DFR] property. Next, we provide a result for the comparison of arrival times of the previous counting processes in the dynamic hazard rate order, when the underlying sequences of random variables are IFR [DFR].

\begin{The}\label{Th3.3}
According to the notation of Theorem \ref{Th3.2}, if $X_i$ has an absolutely continuous IFR [DFR] distribution, for all $i\ge 1$, then
\[
(T'_1,T'_2, \ldots,T'_n)\ge_{\textup{dyn-hr}}[\le_{\textup{dyn-hr}}](T_1,T_2, \ldots,T_n),\text{ for all }n\ge 1.
\]
\end{The}

\begin{proof} Let us assume that $X_i$ is  IFR, for all $i\ge 1$, the DFR case follows similarly by reversing the inequalities. 

Let us consider a fixed $t>0$ and $n>1$ (the case $n=1$ is trivial). Let us denote the multivariate dynamic hazard rates of  $(T'_1,T'_2, \ldots,T'_n)$ and $(T_1,T_2, \ldots,T_n)$ by $\lambda_\cdot(\cdot|\cdot)$ and $\eta_\cdot(\cdot|\cdot)$, respectively. 

Given that 
\[
T'_1 < T'_2<  \ldots < T'_n, \text{ (a.s.)}\]
and 
\[
T_1 < T_2 < \ldots <T_n, \text{ (a.s.)},
\]
then, as we have observed in the previous section, the histories $h_t$ and $h'_t$ have the following form
\[
h_t=\{T_1=t_1,\ldots,T_j=t_j, T_{j+1}>t, \ldots, T_n>t\}
\]
and 
\[
h'_t=\{T'_1=t'_1,\ldots,T'_i=t'_i, T'_{i+1}>t, \ldots, T'_n>t\},
\]
where $i\le j$, $0<t_1<\cdots<t_{j}$, $0<t'_1<\cdots<t'_{i}$ and $t_k \le t'_k < t$ for all $k=1,\ldots,i$. Recall that $h_t$ has to be more severe than $h'_t$.

Now, denoting by $r_i$ the hazard rate function associated to $X_i$, we have that
\[
\lambda_k(t|h'_t)=\left\{ \begin{array}{c l}
r_k(t - t'_i) & \text{if }k=i+1\\
0 & \text{if }k>i+1
\end{array}\right.
\]
and
\[
\eta_k(t|h_t)=\left\{ \begin{array}{c l}
r_k(t) & \text{if }k=j+1 \\
0 & \text{if }k>j+1
\end{array}\right. .
\]
 
Therefore, if $X_i$ is IFR, it is easy to see that 
\[
\eta_{k}(t|h_t)\ge\lambda_{k}(t|h'_t),\text{ for all }k=j+1,\ldots,n.
\]
\end{proof}

In the following results, we consider the same problem as in the previous theorems but considering that the two policies are based on different sequences of distribution functions.

\begin{The}\label{Th3.4}
Let us consider the counting process $\{N'(t); t\ge 0 \}$ corresponding to the replacement by a new one policy,  defined from a sequence of independent continuous random variables $\{X_n; n\ge 1\}$ and the EPB process $\{N(t); t\ge 0 \}$ based on a sequence of distribution functions $\{F_n; n\ge 1\}$ corresponding to a sequence of independent continuous random variables $\{Y_n; n\ge 1\}$. Let us denote by $\{T'_n; n\ge 1\}$ and $\{T_n; n\ge 1\}$ the corresponding arrival times. If 
\begin{itemize} 
\item[a)] $X_1\ge_{\textup{st}}[\le_{\textup{st}}]Y_1$  and

\item[b)]  $X_n \ge_{\textup{st}}[\le_{\textup{st}}]\{Y_n-t|Y_n>t\}$ for all $n=2,\ldots$ and $t>0$ 
\end{itemize}
then, 
\[
(T'_1,T'_2, \ldots,T'_n)\ge_{\textup{st}}[\le_{\textup{st}}](T_1,T_2, \ldots,T_n)\text{ for all }n\ge 1,
\]
and 
$N'(t) \le_{\textup{st}}[\ge_{\textup{st}}]N(t)$, for all $t\ge 0$.
\end{The}

\begin{proof} The proof follows under similar arguments to the one  of Theorem \ref{Th3.2}, but in this case we have
\[
T'_1\overset{\text{d}}{=} X_1\text{ and } T_1\overset{\text{d}}{=} Y_1,
\]
and for all $1<i \le n$
\[
[T_i|T_1=t_1<\cdots<T_{i-1}=t_{i-1}]\overset{\text{d}}{=} [Y_{i}|Y_{i}>t_{i-1}]
\]
and 
\[
[T'_i|T'_1=t_1<\cdots<T'_{i-1}=t_{i-1}]\overset{\text{d}}{=} X_i+t_{i-1}.
\]
for all $0<t_1<\cdots<t_{i-1}$. 
\end{proof}

To finish, we provide an extension of Theorem \ref{Th3.3}.

\begin{The}\label{Th3.5} According to the notation of Theorem \ref{Th3.4}, if for all $i \ge 1$, $X_i$ and $Y_i$ have absolutely continuous distributions such that 
\begin{itemize}
\item[a)] $X_1\ge_{\textup{hr}}[\le_{\textup{hr}}]Y_1$  
and 

\item[b)] $X_n\ge_{\textup{hr}} [\le_{\textup{hr}}]\{Y_n-t|Y_n>t\}$, for all $n=2,\ldots$ and $t>0$, 
\end{itemize}
then, 
\[
(T'_1,T'_2, \ldots,T'_n)\ge_{\textup{dyn-hr}}[\le_{\textup{dyn-hr}}](T_1,T_2, \ldots,T_n)\text{, for all }n\ge 1.
\]
\end{The}

\begin{proof} The proof follows under similar arguments to the one  of Theorem \ref{Th3.3}. We only need to take into account that, given  two histories $h_t$ and $h'_t$ like in the proof of Theorem \ref{Th3.3}, and dynamic hazard rates $\lambda_.(\cdot|\cdot)$ and $\eta_.(\cdot|\cdot)$ for $(T_1,T_2, \ldots,T_n)$ and $(T'_1,T'_2, \ldots,T'_n)$, respectively, we have  
\[
\lambda_k(t|h'_t)=\left\{ \begin{array}{c l}
r_k(t - t'_i) & \text{if }k=i+1\\
0 & \text{if }k>i+1
\end{array}\right.
\]
and
\[
\eta_k(t|h_t)=\left\{ \begin{array}{c l}
s_k(t) & \text{if }k=j+1 \\
0 & \text{if }k>j+1
\end{array}\right. ,
\]
where $i<j$ and $s_k$ and $r_k$ denote the hazard rates associated to $X_k$ and $Y_k$, respectively, for $k=1,\ldots,n$.
\end{proof}

\section{Applications}\label{Sec4}

From previous results we can provide comparisons for the particular cases of minimal repair and generalized Yule birth processes. Next,  we describe these two models and some results that can be derived from the previous results.

\subsection{Minimal repair policies}

When we perform the relevation among two equally distributed units we get the case of a minimal repair, where an item upon failure,  is restored to a working condition just prior to the failure. This concept was introduced by  Barlow and Hunter (1960) and further references and results can be found in Ascher and Feingold (1984). Therefore, minimal repair is  a particular case of relevation when the two items are equally distributed. 

Analogously, if we consider a unit with survival function $\overline F$ which is being continuously minimally repaired, the sequence of failure times define the minimal repair process. Therefore, the minimal repair process is a particular case of the relevation counting process when the sequence of independent random variables $\{X_n; n\ge 1\}$ satisfies that they are identically distributed with common survival function $\overline F$. In this case, the sequence of failure or arrival times is denoted by $\{T_n^{MR};n\ge 1\}$.We also observe that, in this case, the corresponding process $N'$ based on $\{X_n; n\ge 1\}$, is a renewal process where the counting process and the arrival times will be denoted by $\{N^{Ren};t\ge 0\}$ and $\{T^{Ren};n\ge 1\}$, respectively.

From theorems \ref{Th3.2} and \ref{Th3.3}, the following particular results are derived to compare the counting process for a unit which is being continuously minimally repaired with the one derived from a renewal policy for the same unit.

\begin{The}
Let us consider a unit with random lifetime $X$. Let us denote by $\{T^{Ren}_n;n\ge 1\}$ and $\{T^{MR}_n;n\ge 1\}$ the replacement times for the unit under a renewal and minimal repair policies, respectively. 
\begin{itemize} 

\item[a)] If $X$ is  NBU \textup{[}NWU\textup{]}, then, 
\[
(T^{Ren}_1,T^{Ren}_2, \ldots,T^{Ren}_n)\ge_{\textup{st}}[\le_{\textup{st}}](T^{MR}_1,T^{MR}_2, \ldots,T^{MR}_n),\text{ for all }n\ge 1,
\]
and 
\[
N^{Ren}(t) \le_{\textup{st}} [\ge_{\textup{st}}] N^{MR}(t)\text{, for all }t\ge 0,
\]
where $N^{Ren}$ and $N^{MR}$ are the counting processes associated to the arrival times $\{T^{Ren}_n;n\ge 1\}$ and $\{T^{MR}_n;n\ge 1\}$, respectively. 

\item[b)] If $X$ has an absolutely continuous IFR [DFR] distribution, then, 
\[
(T^{Ren}_1,T^{Ren}_2, \ldots,T^{Ren}_n)\ge_{\textup{dyn-hr}}[\le_{\textup{dyn-hr}}](T^{MR}_1,T^{MR}_2, \ldots,T^{MR}_n),\text{ for all }n\ge 1.
\]

\end{itemize}
\end{The}

Therefore, depending on the kind of ageing of the unit under repair, it is better to perform a minimal repair policy or a renewal one.

The previous result can be combined with previous results in the literature to provide some additional comparisons for the minimal repair process. 

For example, Barlow and Proschan (1975) proved that, given a non negative NBU random variable $T$, then 
\[
N^{Ren}(t) \ge_{\textup{st}} N^{A,K}(t)\text{, for all }t\ge 0,
\]
where $N^{A,K}(t)$ denotes the number of failures in $[0,t]$ under an age replacement policy with replacement interval $K$ (for details see Barlow and Proschan, 1975).

This result combined with the previous theorem provides the following statement: 

Given a non negative NBU random variable $T$, then 
\[
N^{MR}(t) \ge_{\textup{st}} N^{A,K}(t)\text{, for all }t\ge 0,
\]

Next, we illustrate with an example, the previous result.

\begin{Exam} Let us consider a random variable $X$ with survival function given by $\overline F(x) = \exp(-h(x))$, for $x\ge 0$, where 
\begin{equation}\label{surv-nbu}
h(x)= \left\{ 
\begin{array}{ l l }
\sin^2(x) & \text{if } x\in\left[0,\frac{\pi}{2}\right] \\
\frac{\pi}{2}\left(x-\frac{\pi}{2}\right) + 1 & \text{if }x\in\left(\frac{\pi}{2},+\infty\right) .\\
\end{array}\right.
\end{equation}

This survival function was introduced by Stoyanov (2013) as an example of a random variable which is NBU but it is not IFR. According to the previous conclusion, the arrival times of the minimal repair process should be smaller, in the stochastic order, than the one of the age process. In Figure \ref{example-age} we can see this fact for $K=0.5$, $1$ and $2$, respectively, for the arrival times 1 to 4.

\begin{figure}[hbt]\centering
\includegraphics[width=\linewidth]{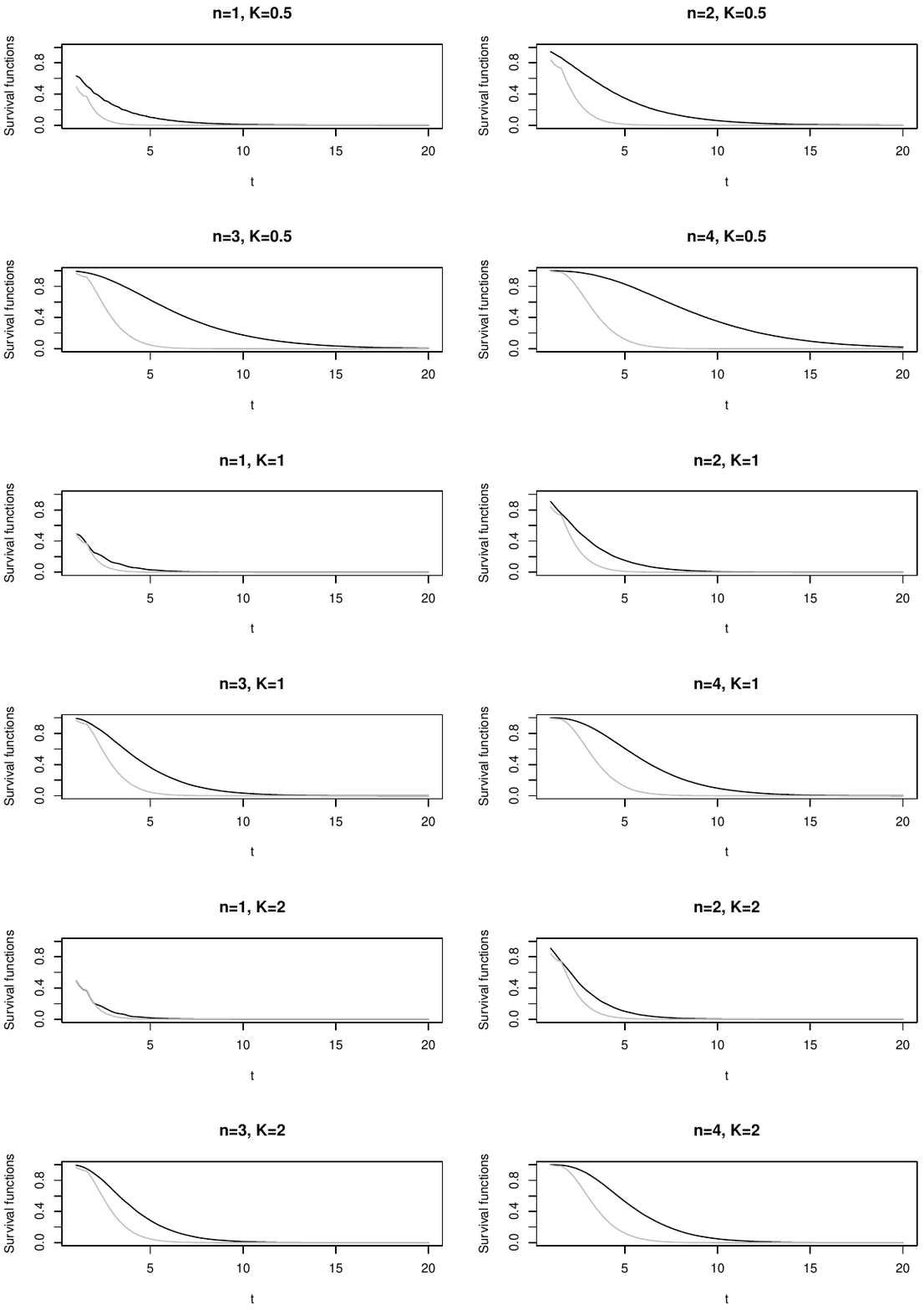}
\caption{\label{example-age}\textit{Survival functions for the arrival times 1 to 4, of a minimal repair process (blue) and an age process for $K=0.5$, 1 and 2 (black), according to \eqref{surv-nbu}.}}
\end{figure}
\end{Exam}

Several additional results can be given in a similar way for some other replacement policies like block  and age-dependent block replacement policies (see Barlow and Proschan, 1975 and Li, 2005).

\subsection{Generalized Yule birth processes}

As we have mentioned in the introduction, another particular case of an EPB process is the generalized Yule birth process. A Yule birth process is a pure birth process with jump intensity of the form $(i+1)r$, from state $i$ to state $i+1$, where $r>0$.  A generalization of the previous stochastic process occurs when  the jump intensity from state $i$ to $i+1$ is given by $(i+1)r(t)$ and, therefore, depends on the calendar time $t$.

This generalized Yule birth process is a particular case of an EPB process for a sequence of random variables $X_i's$ with hazard rates $(i+1)r(t)$, respectively. 

Let us denote by $\{T_n^{Y};n\ge 1\}$, the sequence of arrival times of a generalized Yule birth process based on the sequence of distribution functions associated to the sequence of random variables $\{X_n; n\ge 1\}$, with hazard rates $(n+1)r(t)$, for $n\ge 1$, and by $\{T'_n; n\ge 1\}$ the corresponding arrival times where a unit upon failure is replaced by a new one, based on the sequence of independent random variables $\{X_n; n\ge 1\}$.

Therefore, according to Theorem \ref{Th3.3}, if $r(t)$ is increasing [decreasing] then 
\[
(T'_1,T'_2, \ldots,T'_n)\ge_{\textup{dyn-hr}}[\le_{\textup{dyn-hr}}](T^{Y}_1,T^{Y}_2, \ldots,T^{Y}_n),\text{ for all }n\ge 1.
\]

\section{Conclusions}\label{Sec5}

In this paper, we study the comparison of relevation and replacement by a new one unit policies in terms of the failure times under the corresponding repair policy,  and we show that, depending on the ageing properties of the units under repair, one of the policies is better than the other one. These results are extended to the comparison of the previous counting processes, based on different sequences of random variables. The comparisons are made in terms of the multivariate stochastic and hazard rate orders. 

\section*{Acknowledgements} The authors want to acknowledge the comments by the Associate Editor and two anonymous referees that have greatly improved the presentation and the contents of this paper. The authors also want to acknowledge the support received from the Ministerio de Economía, Industría y Competitividad under grant MTM2016-79943-P (AEI/FEDER, UE).

\end{document}